\PassOptionsToPackage{dvipsnames}{xcolor}

\documentclass[11pt]{amsart}
\usepackage{amsmath, amssymb, graphicx, url, amsthm, multicol, latexsym, enumerate, bbm, mathtools, physics, microtype, cite, xcolor, float}
\usepackage{units}
\usepackage[all,cmtip]{xy}
\usepackage{wrapfig}
\usepackage[hmargin=1.25in, vmargin=1.25in]{geometry} 
\usepackage{tikz}
\usepackage{ytableau}
\usetikzlibrary{shapes.geometric, arrows}

\usepackage{hyperref}
\hypersetup{
    colorlinks=true,
    linkcolor=blue,
    citecolor=magenta,
    filecolor=magenta,      
    urlcolor=magenta,
}

\newtheorem{theorem}{Theorem}[section]
\newtheorem{proposition}[theorem]{Proposition}

\newtheorem{corollary}[theorem]{Corollary}

\theoremstyle{definition}

\newtheorem{example}[theorem]{Example}

\theoremstyle{remark}

\newcommand{\id}{\mathrm{id}}

\newcommand{\sqbinom}{\genfrac{[}{]}{0pt}{}}

\makeatletter % to repeat theorem numbers
\newtheorem*{rep@theorem}{\rep@title}\newcommand{\newreptheorem}[2]{%
\newenvironment{rep#1}[1]{%
\def\rep@title{\bf #2 \ref{##1}}%
\begin{rep@theorem}}%
{\end{rep@theorem}}}
\makeatother
\newreptheorem{theorem}{Theorem}

\usepackage[colorinlistoftodos]{todonotes}

\definecolor{munsell}{rgb}{0.0, 0.5, 0.69}

\begin{document}

%%%%%%%%%%%%%%%%%%%%%%%%%%%%%%%%%%%%%%%

\title{Combinatorial results on barcode lattices}

\author{Alex Bouquet}
\address{\scriptsize{Department of Mathematics, UC Berkeley}}
\email{\scriptsize{abouquet@berkeley.edu}}

\author{Andr\'es R. Vindas-Mel\'endez}
\address{\scriptsize{Departments of Mathematics, UC Berkeley \& Harvey Mudd College}, \url{https://math.berkeley.edu/~vindas}}
\email{\scriptsize{andres.vindas@berkeley.edu; arvm@hmc.edu}}

%%%%%%%%%%%%%%%%%%%%%%%%%%%%%%%%%%%%%%%

\begin{abstract}
A barcode is a finite multiset of intervals on the real line. 
Jaramillo-Rodriguez (2023) previously defined a map from the space of barcodes with a fixed number of bars to a set of multipermutations, which presented new combinatorial invariants on the space of barcodes.
A partial order can be defined on these multipermutations, resulting in a class of posets known as combinatorial barcode lattices.
In this paper, we provide a number of equivalent definitions for the combinatorial barcode lattice, show that its M\"obius function is a restriction of the M\"obius function of the symmetric group under the weak Bruhat order, and show its ground set is the Jordan-H\"older set of a labeled poset.
Furthermore, we obtain formulas for the number of join-irreducible elements, the rank-generating function, and the number of maximal chains of combinatorial barcode lattices. 
Lastly, we make connections between intervals in the combinatorial barcode lattice and certain classes of matchings.
\end{abstract}

%%%%%%%%%%%%%%%%%%%%%%%%%%%%%%%%%%%%%%%

\maketitle

%%%%%%%%%%%%%%%%%%%%%%%%%%%%%%%%%%%%%%%

\section{Introduction}

A \emph{barcode} is a finite multiset of closed intervals on the real number line.
Barcodes appear in different contexts, including: topological data analysis as summaries of the persistent homology groups of a filtration \cite{ZomorodianCarlsson} and in graph theory as interval graphs \cite{LekkerkerkerBoland}. 
Recently, in \cite{Jaramillo-RodriguezThesis} and \cite{jaramillorodriguez2022barcode}, Jaramillo-Rodriguez developed combinatorial methods for analyzing barcodes for applications in topological data analysis and random interval graphs.

Jaramillo-Rodriguez introduced a map from the space of barcodes to certain equivalence classes of permutations of a multiset in which every element occurs exactly twice, which she calls \emph{double occurrence words}. 
Futhermore, she calls the set of all such words the space of \emph{combinatorial barcodes}.
By defining an order relation on this space that is based on the weak-Bruhat order, the resulting poset was shown to be a graded lattice, which is referred to as the \emph{combinatorial barcode lattice}.
In particular, the cover relations of this lattice were used to determine the set of barcode bases of persistence modules, which arise in topological data analysis.
While being of interest from a topological perspective, we focus on obtaining combinatorial results and thus treat the barcode lattice solely as a combinatorial object.

The paper is structured as follows:
\begin{itemize}
    \item Section \ref{sec:background} presents relevant background on combinatorial barcode lattices.

    \item We continue with Section \ref{sec:initial_results}, which presents some initial counting results on the combinatorial barcode lattice. 
    In particular, we determine the lattice's M\"obius function (Proposition \ref{prop:mobius}) and the lattice's number of join-irreducible elements (Theorem \ref{thm:join_irred}).
    We also realize the combinatorial barcode lattice as a partial order on the set of linear extensions of a labeled poset, connecting our work to a paper of Bj\"orner and Wachs \cite{bjornerWachsLinearExtensions}.

    \item In Section \ref{sec:rank_gen_func}, we determine the rank-generating function of the combinatorial barcode lattice (Theorem \ref{thm:rank_gen_func}) by presenting a different lattice which is more manageable to work with and show that it has the same rank-generating function as the combinatorial barcode lattice.

    \item Section \ref{sec:max_chains} is devoted to determining the number of maximal chains in the combinatorial barcode lattice (Theorem \ref{thm:max_chains}).

    \item Connections between the combinatorial barcode lattice and matchings are presented in Section \ref{sec:matchings}.

    \item Lastly, we conclude with further directions for future work in Section~\ref{sec:further_directions}.
    
\end{itemize}
%%%%%%%%%%%%%%%%%%%%%%%%%%%%%%%%%%%%%%%%%

\section{Background \& Preliminaries}\label{sec:background}

The original combinatorial barcode lattice $L(n,2)/\mathfrak S_n$ defined by Jaramillo-Rodriguez in \cite{jaramillorodriguez2022barcode} is a lattice whose elements are multipermutations of the multiset $\{\{1,1,2,2,\dots,n,n\}\}$ such that the first appearance of the element $i$ in the multipermutation appears before the first appearance of the element $i+1$ for each $i$.
Jaramillo-Rodriguez's power-$k$ barcode lattice $L(n,2^k+1)/\mathfrak S_n$ is defined similarly, but for multipermutations of the multiset
\[\{\{\underbrace{1,\dots,1}_{2^k+1},\underbrace{2,\dots,2}_{2^k+1},\dots,\underbrace{n,\dots,n}_{2^k+1}\}\}.\]

To develop a more general combinatorial theory, we relax the definition of the combinatorial barcode lattice.
We allow the multiplicities of elements in the multiset to be any positive integer and we do not require all multiplicities to be the same.
With this in mind, for $\mathbf{m}=(m_1,\dots,m_n)$, we define $\mathbf{BL}(\mathbf{m})$ to be the set of multiset permutations of the multiset $\{\{1^{m_1},\dots,n^{m_n}\}\}$, where for $i\in[n-1]$, we require the first occurrence of $i$ in the permutation to appear before the first occurrence of $i+1$.
We call the elements of this set ``\emph{barcodes}'' and we partially order them by declaring that a barcode $t$ covers a barcode $s$ (written $s\lessdot t$) if and only if $t$ differs from $s$ by a transposition of two distinct adjacent elements which are increasing in $s$ and decreasing in $t$.
We call the resulting poset the combinatorial barcode lattice.
It will follow from Proposition \ref{principal ideal} that $\mathbf{BL}(\mathbf m)$ is, in fact, a lattice.
(See Figure \ref{BL 2,2,2} for the Hasse diagram of a combinatorial barcode lattice.)

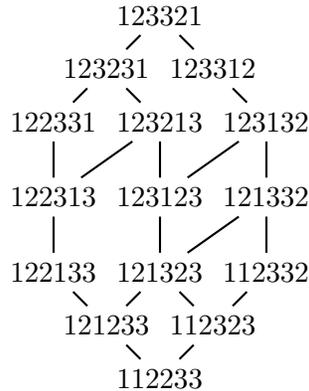
\begin{figure}[h]
\begin{center}
\begin{tikzpicture}[scale=0.7][node distance = 12mm and 12mm]
    % Rank 0
    \node (112233) at (0,0) {112233};

    % Rank 1
    \node [above left of=112233] (121233) {121233};
    \node [above right  of=112233] (112323)  {112323};
    
    \draw [thick] (112233) -- (112323);
    \draw [thick] (112233) -- (121233);
    
    % Rank 2
    \node [above left of=121233] (122133) {122133};
    \node [above right of=121233] (121323) {121323};
    \node [above right of=112323] (112332) {112332};
    
    \draw [thick] (121233) -- (122133);
    \draw [thick] (121233) -- (121323);
    
    \draw [thick] (112323) -- (121323);
    \draw [thick] (112323) -- (112332);

    % Rank 3
    \node [above of=122133] (122313) {122313};
    \node [above of=121323] (123123) {123123};
    \node [above of=112332] (121332) {121332};
    
    \draw [thick] (122133) -- (122313);
    
    \draw [thick] (121323) -- (123123);
    \draw [thick] (121323) -- (121332);

    \draw [thick] (112332) -- (121332);

    % Rank 4
    \node [above of=122313] (122331) {122331};
    \node [above of=123123] (123213) {123213};
    \node [above of=121332] (123132) {123132};
    
    \draw [thick] (122313) -- (122331);
    \draw [thick] (122313) -- (123213);
    
    \draw [thick] (123123) -- (123213);
    \draw [thick] (123123) -- (123132);

    \draw [thick] (121332) -- (123132);

    % Rank 5
    \node [above right  of=122331] (123231)  {123231};
    \node [above left of=123132] (123312) {123312};
    
    \draw [thick] (122331) -- (123231);
    
    \draw [thick] (123213) -- (123231);
    
    \draw [thick] (123132) -- (123312);

    %Rank 6
    \node [above right of=123231] (123321) {123321};

    \draw [thick] (123231) -- (123321);
    
    \draw [thick] (123312) -- (123321);
    
\end{tikzpicture}
\end{center}
\caption{The combinatorial barcode lattice $\mathbf{BL}(2,2,2)$.}
\label{BL 2,2,2}
\end{figure}

We see that $\mathbf{BL}(2,2,\dots,2)$ is the standard combinatorial (power-$0$) barcode lattice $L(n,2)/\mathfrak S_n$ and $\mathbf{BL}(2^k+1,2^k+1,\dots,2^k+1)$ is the power-$k$ barcode lattice $L(n,2^k+1)/\mathfrak S_n$ defined in \cite{jaramillorodriguez2022barcode}.
For convenience, we will denote the combinatorial barcode lattice with $n$ bars all of size $k$ by
\[\mathbf{BL}(k^n)=\mathbf{BL}(\underbrace{k,\dots,k}_{n\text{ times}}).\]
Equivalently, we can also write elements of the combinatorial barcode lattice $\mathbf{BL}(\mathbf m)$ as permutations of the totally ordered set
\[1_1,\dots,1_{m_1},2_1,\dots,2_{m_2},\dots,n_1,\dots,n_{m_n},\]
where we require that the subsequence consisting of the entries $1_1,2_1,3_1,\dots,n_1$ appears in exactly that order and that for each $i$ the subsequence consisting of the entries $i_1,i_2,\dots,i_{m_i}$ appears in exactly that order.
For an entry $i_j$ in a barcode $s$, we will refer to $i$ as the label and $j$ as the index.

Then, for $s\in\mathbf{BL}(\mathbf m)$, we define $\Phi_s(i_j)$ to be the number of entries of $s$ that occur before $i_j$ with a label larger than $i$.
For example, \[\text{ if } s=1_12_12_21_22_31_3\in\mathbf{BL}(3^2), \text{ then } \Phi_s(1_2)=1 \text{ and } \Phi_s(1_3)=3.\]
We observe that for any $i$, $\Phi_s(i_1)=0$, because we require the subsequence $1_1,2_1,3_1,\dots,n_1$ to appear in order, we can also think of $\Phi_s(i_j)$ as the number of entries between $i_1$ and $i_j$ with label larger than $i$.
Also, since the subsequence $i_1,i_2,\dots,i_{m_i}$ appears in order, we know that for any $i,j$, we have that $\Phi_s(i_j)\leq\Phi_s(i_{j+1})$.
Lastly, we observe that for $s\in\mathbf{BL}(m_1,\dots,m_n)$, there can be at most $\sum_{k=i+1}^nm_k$ entries between $i_1$ and $i_j$ whose label is larger than $i$.

%%%%%%%%%%%%%%%%%%%%%%%%%%%%%%%%%%%%%%%%%

\section{Initial counting results on the combinatorial barcode lattice}\label{sec:initial_results}

We now present a number of initial observations about the combinatorial barcode lattice $\mathbf{BL}(\mathbf m)$.
We begin by realizing that $\mathbf{BL}(\mathbf m)$ is a principal order ideal of the symmetric group under the weak Bruhat order, which immediately gives us the M\"obius function of $\mathbf{BL}(\mathbf m)$ and a characterization of when $\mathbf{BL}(\mathbf m)$ is distributive.
We then find the number of join-irreducible elements of $\mathbf{BL}(\mathbf m)$ and conclude this section by realizing that the combinatorial barcode lattice is a partial order on linear extensions of another poset.

Let us start by recalling that the symmetric group $S_n$ forms a lattice under the weak Bruhat order, which can be defined by the covering relation $\sigma\prec\tau$ if and only if $\sigma$ and $\tau$ (written in one line notation) differ exactly by a transposition of adjacent entries that appear in order in $\sigma$ and reversed in $\tau$.
We begin by noticing that the combinatorial barcode lattice $\mathbf{BL}(\mathbf m)$ is a principal order ideal of the symmetric group under the weak Bruhat order:

\begin{proposition}\label{principal ideal} The combinatorial barcode lattice $\mathbf{BL}(m_1,\dots,m_n)$ with $\sum_{i=1}^nm_i=M$ is isomorphic to a principal order ideal of the symmetric group $S_M$ in the weak Bruhat order.
\end{proposition}

For ease of notation, define $M_i=(\sum_{k=1}^{i}m_i)$ so that $M_n=M$, then break the set $[M]$ into $n$ blocks:
\begin{align*}
    1,&\dots,M_1 &(B_1)\\
    M_1+1,&\dots,M_2 &(B_2)\\
    M_2+1,&\dots,M_3 &(B_3)\\
    &\vdots&\\
    M_{n-1}+1,&\dots,M_n &(B_n).
\end{align*}
We claim that $\mathbf{BL}(m_1,\dots,m_n)$ is isomorphic to the principal order ideal of $S_M$ generated by the permutation given in one line notation by reading the first entry of each block $B_i$ in increasing order, then reading the remaining entries of $B_n$ in increasing order, then the remaining entries of $B_{n-1}$ in increasing order, and so on, until all of the blocks have been read off.
We denote this permutation as $\beta$.

\begin{proof}
Note that any element $s$ of the principal order ideal $[1^{m_1}\cdots n^{m_n},\beta]$ will have all entries from the same block appearing in numerical order, with the first entry of one block appearing before each entry of subsequent blocks.
Make the identification \[M_i+j\mapsto (i+1)_j,\] so that entries from the same block have the same label, now we see that $s$ is a barcode in $\mathbf{BL}(m_1,\dots,m_n)$.
Additionally, any barcode in $\mathbf{BL}(m_1,\dots,m_n)$ must be less than or equal to the fully reversed barcode $1_1\dots n_1n_2\dots n_{m_n}\dots1_2\dots1_{m_1}$ (in the weak Bruhat order), which is identified with $\beta$.
This identification is order-preserving and order-reflecting, which follows from the fact that the two lattices have the same covering relation.
\end{proof}

This immediately gives us our next two results on the combinatorial barcode lattice:

\begin{proposition}\label{prop:mobius}
 If $\sum_{i=1}^nm_i=M$, then the M\"obius function of $\mathbf{BL}(\mathbf m)$ is a restriction of the M\"obius function of the symmetric group $S_M$ with the weak Bruhat order.
That is, the M\"obius function of $\mathbf{BL}(\mathbf m)$ is
\[\mu(s,t)=\begin{cases}
    (-1)^{|J|}&\text{if }t=sw_0(J)\text{ for some } J\subseteq S_M,\\
    0&\text{otherwise},
\end{cases}\]
where $w_0(J)$ is the top element of the subgroup of $S_M$ generated by $J$.
\end{proposition}

\begin{proof}
Since $\mathbf{BL}(\mathbf m)$ is isomorphic to an interval in $S_M$ and the M\"obius function of an interval in a poset is the same as the M\"obius function of the entire poset restricted to the interval, the proposition follows.
\end{proof} 

\begin{corollary}
Let $\mathbf m=(m_1,\dots,m_n)$.
Then $\mathbf{BL}(\mathbf m)$ is distributive if and only if \[\#\{i\mid m_i\geq 2\}\leq 2.\]
\end{corollary} 

\begin{proof}
It is known that a the principal order ideal $[e,w]\subseteq S_n$ in the weak Bruhat order is a distributive lattice if and only if $w$ is 321-avoiding.
For more results on pattern avoidance and intervals in the Bruhat order, see \cite{TennerIntervals}.
Equivalently, $[e,w]$ is a distributive lattice if and only if the length of the longest decreasing subsequence of $w$ is less than three.
Taking $S_n$ to act on the set $\{1_1,\dots,1_{m_n},\dots,n_1,\dots,n_{m_n}\}$, we write $\beta$ as
\begin{align*}
    1_12_1&\cdots n_1\\
    n_2&\cdots n_{m_n}\\
    &\vdots\\
    2_2&\cdots2_{m_2}\\
    1_2&\cdots1_{m_1}.
\end{align*}
Then, it is discerned that a maximum-length decreasing subsequence of $\beta$ is
\[n_2,(n-1)_2,\dots,2_2,1_2.\]
The length of this subsequence is less than three if and only if there are fewer than three $i$ such that $m_i\geq2$.
\end{proof}

For our first entirely enumerative result, we will determine the number of join-irreducible elements of the combinatorial barcode lattice:

\begin{theorem}\label{thm:join_irred}
For $\mathbf m=(m_1,\dots,m_n)$, the number of join-irreducible elements in $\mathbf{BL}(\mathbf m)$ is
\[\prod_{i=1}^n(m_i+1)-(m_1+m_2+\dots+m_n+1)-\sum_{i=1}^{n-1}\left[\left(\prod_{j=1}^{i-1}m_j\right)\left(\biggl(\prod_{j=i+1}^nm_j+1\biggr)-1\right)\right].\]
\end{theorem} 

\begin{figure}[h]
\begin{center}
\begin{tikzpicture}[scale=0.9][node distance = 12mm and 12mm,every edge quotes/.style = {auto, font=\footnotesize, sloped}]

\node (1223113) at (0,0) {1223113};
\node[below right of=1223113,gray] (1221313) {1221313};

\node[above right of=1221313, gray] (1221331) {1221331};
\node[above of=1223113,gray] (1232113) {1232113};
\node[above of=1221331,gray] (1223131) {1223131};
\node[below left of=1221313,gray] (1212313) {1212313};
\node[below of=1212313] (blankBR) {\quad};
\node[below right of=1221313] (1221133) {1221133};
\node[below of=1212313] (blankBL) {\quad};
\node[below of=1221133,gray] (blankBR) {1212133};
\node[above of=1232113] (blankTL) {\quad};
\node[above of=1223131] (blankTR) {\quad};
\draw[blue,thick] (1223113) -- (1221313);
\draw[gray,dashed] (1212313) -- (1221313);
\draw[gray,dashed] (1221133) -- (1221313);
\draw[gray,dashed] (1223113) -- (1232113);
\draw[gray,dashed] (1223113) -- (1223131);
\draw[gray,dashed] (1232113) -- (blankTL);
\draw[gray,dashed] (1223131) -- (blankTR);
\draw[gray,dashed] (1212313) -- (blankBL);
\draw[blue,thick] (1221133) -- (blankBR);

\draw[gray,dashed] (1221313) -- (1221331);
\draw[gray,dashed] (1221331) -- (1223131);
\node[below right of=1221331] (anotherBlank) {\quad};
\draw[gray,dashed] (1221331) -- (anotherBlank);
\node[below left of=1212313] (lastBlank) {\quad};
\draw[gray,dashed] (1212313) -- (lastBlank);
\node[above right of=1223131] (nevermind) {\quad};
\draw[gray,dashed] (1223131) -- (nevermind);
\node[below left of=1232113] (oneMoreThing) {\quad};
\draw[gray,dashed] (1232113) -- (oneMoreThing);

\end{tikzpicture}
\caption{A depiction showing that 1223113 and 1221133 are two join-irreducible elements of $\mathbf{BL}(3,2,2)$.}
\label{ji example}
\end{center}
\end{figure}
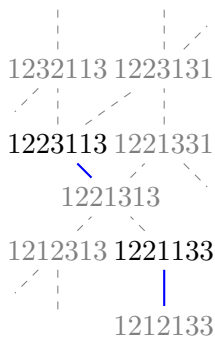

\begin{proof}
Recall that the multinomial Newman lattice $L(\mathbf m)$ defined in \cite{bennettBirkhoffNewmanLattices} is the lattice of strings containing $m_i$ copies of $i$ for each $i\in[n]$, where the covering relation is defined by $s\prec t$ if and only if $s$ and $t$ differ exactly by a transposition of adjacent entries that appear in order in $s$ and reversed in $t$.
Bennett and Birkhoff prove in \cite{bennettBirkhoffNewmanLattices} that the number of join-irreducible elements in $L(\mathbf m)$ is $\prod_{i=1}^n(m_i+1)-(m_1+m_2+\dots+m_n+1)$.
Note that by a similar argument to Proposition \ref{principal ideal}, we have that $\mathbf{BL}(\mathbf m)$ is also a principal order ideal of $L(\mathbf m)$ which is itself a principal order ideal of $S_M$.
We proceed by recounting Bennett and Birkhoff's proof and then subtracting off the extra join-irreducibles that are present in $L(\mathbf m)$, but not in $\mathbf{BL}(\mathbf m)$.

Observe that each join-irreducible element consists of a weakly increasing string followed by a unique descent into another increasing string.
(See Figure \ref{ji example} for an example in $\mathbf{BL}(3,2,2)$.)
This is because if there are two descents in a string \[S=r^\smallfrown ba^\smallfrown s^\smallfrown dc^\smallfrown t,\] where $r$, $s$, and $t$ are strings, $a,b,c$ and $d$ are elements of the ground set such that $a<b$ and $c<d$, and $^\smallfrown$ denotes concatenation of strings, then we can express $S$ as
\[S=r^\smallfrown ab^\smallfrown s^\smallfrown dc^\smallfrown t\vee r^\smallfrown ba^\smallfrown s^\smallfrown cd^\smallfrown t.\]
Since a string with a unique descent consists of a weakly increasing substring before the descent and a weakly increasing substring after the descent, we can uniquely specify a join-irreducible element by how many entries with each label come before the adjacent inversion.
Thus, we can represent a join-irreducible element of $L(\mathbf m)$ by the vector $\mathbf x=(x_1,\dots,x_n)$, where $x_i$ is the number of entries $i_j$ appearing before the descent.
There are $\prod_{i=1}^n(m_i)$ vectors $(x_1,\dots,x_n)$ with $0\leq x_i\leq m_i$ for each $i$.

We must now subtract off all vectors that do not correspond to a join-irreducible element.
We first subtract all vectors of the form $(m_1,\dots,m_i,x_{i+1},0,\dots,0)$, since such a vector corresponds to a string having initial segment $1^{m_1}2^{m^2}\cdots i^{m_i}(i+1)^{x_{i+1}}k$ with $k<i+1$.
This cannot be, since each $k$ with $k<i+1$ appears before the first occurrence of $i+1$.
We also subtract the zero vector because a minimal element of a lattice does not count as a join-irreducible element.
This leaves us with $\prod_{i=1}^n(m_i+1)-(m_1+m_2+\dots+m_n+1)$ join-irreducible elements in $L(\mathbf m)$.

We continue by subtracting the vectors that are join-irreducible elements of $L(\mathbf m)$, but are not join-irreducible elements of $\mathbf{BL}(\mathbf m)$.
These vectors correspond to strings $s$ containing entries $i_1$ and $j_1$ with $i<j$ and $j_1$ appearing before $i_1$ in $s$.
Since we are only considering strings that have a unique descent, we know that $j_1$ must appear before the descent and $i_1$ must appear after the descent.
If $j_1$ appears before the descent, we know that the vector $\mathbf x$ corresponding to this element must have $x_j>0$, and if $i_1$ appears after the descent, we must have $x_i=0$.
Thus, a vector $(x_1,\dots,x_n)$ corresponding to a join-irreducible element of $L(\mathbf m)$ is also a join-irreducible element of $\mathbf{BL}(\mathbf m)$ if whenever $x_i$ is positive for some $i$, $x_j$ is positive for all $j<i$.

Now, we count all vectors having $x_i=0$ and $x_j>0$ for some $i<j$.
Let $i$ be the minimal such that $x_i=0$.
Then, for each $j<i$, $x_j$ can be anything greater than $0$, and for each $k>i$, $x_k$ can be anything as long as there is at least one $k>i$ with $x_k>0$.
Taking the sum over all $i$ less than $n$ gives us
\[\sum_{i=1}^{n-1}\left[\left(\prod_{j=1}^{i-1}m_j\right)\left(\biggl(\prod_{j=i+1}^nm_j+1\biggr)-1\right)\right]\]
vectors to exclude.
This gives us a total of
\[\prod_{i=1}^n(m_i+1)-(m_1+m_2+\dots+m_n+1)-\sum_{i=1}^{n-1}\left[\left(\prod_{j=1}^{i-1}m_j\right)\left(\biggl(\prod_{j=i+1}^nm_j+1\biggr)-1\right)\right]\]
join-irreducible elements in $\mathbf{BL}(\mathbf m)$, as was to be proved.
\end{proof}

\noindent To see the number of join-irreducible elements in some of the power-$k$ barcode lattices, see Table \ref{JI Table}.

\begin{center}
\begin{table}[h]
\begin{tabular}{|c|c||c|c||c|c|}
    \hline
    $\quad$ & Number of & $\quad$ & Number of & $\quad$ & Number of\\
    $(c^n)$ & join-irreducible&$(c^n)$ & join-irreducible&$(c^n)$ & join-irreducible\\
    $\quad$ & Elements in $\mathbf{BL}(c^n)$&$\quad$ & Elements in $\mathbf{BL}(c^n)$&$\quad$ & Elements in $\mathbf{BL}(c^n)$\\
    \hline\hline
    $2^2$ & 2&$3^2$&6&$5^2$&20\\
    \hline
    $2^3$ & 8&$3^3$&30&$5^3$&140\\
    \hline
    $2^4$ & 22&$3^4$&108&$5^4$&760\\
    \hline
    $2^5$ & 52&$3^5$&348&$5^5$&3880\\
    \hline
    $2^6$ & 114&$3^6$&1074&$5^6$&19500\\
    \hline
    \hline
    $9^2$ & 72 &$17^2$&272&$33^2$&1056\\
    \hline
    $9^3$ & 792 &$17^3$&5168&$33^3$&36960\\
    \hline
    $9^4$ & 7344 &$17^4$&88672&$33^4$&1222848\\
    \hline
    $9^5$ & 66384 &$17^5$&1508512&$33^5$&40358208\\
    \hline
    $9^6$ & 597816 &$17^6$&25646064&$33^6$&1331826144\\
    \hline
    
\end{tabular}
\caption{Numbers of join-irreducible elements in some power-$k$ barcode lattices.}
\label{JI Table}
\end{table}
\end{center}

Now, we observe that the multinomial Newman lattice and the combinatorial barcode lattice are both examples of a more general class of structures.

\begin{proposition}
Denote by $\mathbf{\hat{n}}$ the chain $1<2<\dots<n-1<n$.
Then, for $\mathbf{m}=(m_1,\dots,m_n)$, let $P(\mathbf m)$ denote the poset consisting of the $n+1$ chains $\mathbf{\hat{m}_1},\dots,\mathbf{\hat{m}_n}, \mathbf{\hat{n}}$, where for each $i$ the minimum element of $\mathbf{\hat{m}_i}$ is identified with $i$ in the chain $\mathbf{\hat{n}}$.
The combinatorial barcode lattice $\mathbf{BL}(\mathbf m)$ is a partial order on the set of linear extensions of the poset $P(\mathbf m)$.
\end{proposition}

\begin{proof}
We obtain the poset $P(\mathbf m)$ by partially ordering the set \[S=\{1_1,\dots,1_{m_1},\dots,n_1,\dots,n_{m_n}\}\] with the order $\prec$ defined by
\[a_i\prec b_j\Leftrightarrow(a\leq b \wedge i=1)\vee(i\leq j\wedge a=b).\]
Then we see that the permutation of the labels of $S$ obtained from a linear order on $S$ by listing 
 the elements of $S$ in order corresponds to a barcode if and only if the linear order extends $\prec$.   
\end{proof}

Note that in general, if we have a poset $P$ with $|P|=n$ and a (bijective) labelling $\omega:P\rightarrow [n]$ of $P$ by elements of $[n]$, then the linear extensions of $P$ can be thought of as elements of the symmetric group $S_n$ by identifying the linear extension $x_1<x_2<\cdots<x_n$ with the permutation $\omega(x_1)\omega(x_2)\cdots\omega(x_n)$.
We can then consider the induced subposet of the weakly ordered symmetric group consisting of these elements.
We call the resulting poset $\mathcal L(P,\omega)$.
\begin{itemize}
    \item We can see that for the poset $P(\mathbf m)$ defined in the previous proof and taking $\omega$ to be the labelling that sends $i_j$ to $j+\sum_{1\leq k<i}m_i$, then $\mathcal L(P(\mathbf m),\omega)$ is the combinatorial barcode lattice.

    \item If the base poset is instead the antichain $A$ on $[n]$ and $\omega$ is any labelling, then $\mathcal L(A,\omega)$ is the entire symmetric group.

    \item If the base poset is $Q(\mathbf m)$ whose underlying set is $\{i_j\mid j\in[m_i]\}$ and whose order is defined by $i_j\leq k_l \Leftrightarrow i=k\wedge j\leq l$, and $\omega$ is the labelling $\omega(i_j)=j+\sum_{1\leq k<i}m_i$ then $\mathcal{L}(Q(\mathbf m),\omega)$ is the multinomial Newman lattice $L(\mathbf m)$.
\end{itemize}
The (unordered) ground set of $\mathcal L(P,\omega)$ is called the Jordan-H\"older set of $P$.
There is some literature on the set of linear extensions of a labeled poset under the weak Bruhat order, namely Bj\"orner and Wach's paper\cite{bjornerWachsLinearExtensions}.
This paper focuses mainly on the posets $\mathcal L(P,\omega)$ that satisfy
\[\sum_{\sigma\in\mathcal L(P,\omega)}q^{\mathrm{inv}(\sigma)}=\sum_{\sigma\in\mathcal L(P,\omega)}q^{\mathrm{maj}(\sigma)},\]
and proves that equality holds only when the ground poset $P$ is a forest (i.e., every element of $P$ is covered by at most one element) and $w$ is a postorder labeling.
The poset $P$ for which $\mathcal L(P,\omega)=\mathbf{BL}(\mathbf m)$, where $\mathbf{m}=(m_1,m_2,\dots,m_n)$ is only a forest when $m_i=1$ for all $i<n$.
If $\mathbf m$ is of this form, however, $\mathbf{BL}(\mathbf m)$ only contains a single element.
We can conclude that in all cases where $\mathbf{BL}(\mathbf m)$ has more than one element
\[\sum_{\sigma\in\mathbf{BL}(\mathbf m)}q^{\mathrm{inv}(\sigma)}\neq\sum_{\sigma\in\mathbf{BL}(\mathbf m)}q^{\mathrm{maj}(\sigma)}.\]
We realize that the left hand side of the equation above is the rank-generating function of $\mathbf{BL}(\mathbf m)$.
To see this, notice that for $s,t\in\mathbf{BL}(\mathbf m)$ we have that $t$ covers $s$ if and only if a pair of increasing adjacent elements of $s$ can be swapped to obtain $t$, meaning $\mathrm{inv}(t)=\mathrm{inv}(s)+1$.
Thus, if we have $\mathrm{rk}(t)=r$ then we have a sequence
\[\id=s_0\lessdot s_1\lessdot\dots\lessdot s_r=t,\]
where $\id$ is the identity permutation with $\mathrm{rk}(\id)=\mathrm{inv}(\id)=0$.
Thus,
\[\mathrm{inv}(t)=\mathrm{inv}(s_r)=\mathrm{inv}(s_{r-1})+1=\dots=\mathrm{inv}(\id)+r=r=\mathrm{rk}(t).\]
In what follows, we give an explicit formula for the rank-generating function of $\mathbf{BL}(\mathbf m)$.

%%%%%%%%%%%%%%%%%%%%%%%%%%%%%%%%%%%%%%%

\section{The rank-generating function of  \texorpdfstring{$\mathbf{BL}(\mathbf{m})$}{\textbf{BL}(\textbf{m})}}\label{sec:rank_gen_func}

To obtain the rank-generating function of $\mathbf{BL}(\mathbf m)$, we will define a ``simpler" lattice whose rank-generating function is clear and show that this lattice has the same rank-generating function as $\mathbf{BL}(\mathbf m)$.
Let $V(\mathbf{m})$ be the set of vectors $\langle x_{1_1},\dots,x_{1_{m_1}},\dots,x_{n_1},\dots,x_{n_{m_n}}\rangle$ such that for all $i,j$,
\[0=x_{i_1}\leq x_{i_j}\leq x_{i_{j+1}}\leq\sum_{k=i+1}^nm_k.\]
We partially order $V(\mathbf m)$ by declaring that $v\lessdot v'$ if and only if $v$ and $v'$ differ at exactly one entry, which is one larger in $v'$ than in $v$.
(See Figure \ref{V(4,3)} for the Hasse diagram of $V(4,3)$.)

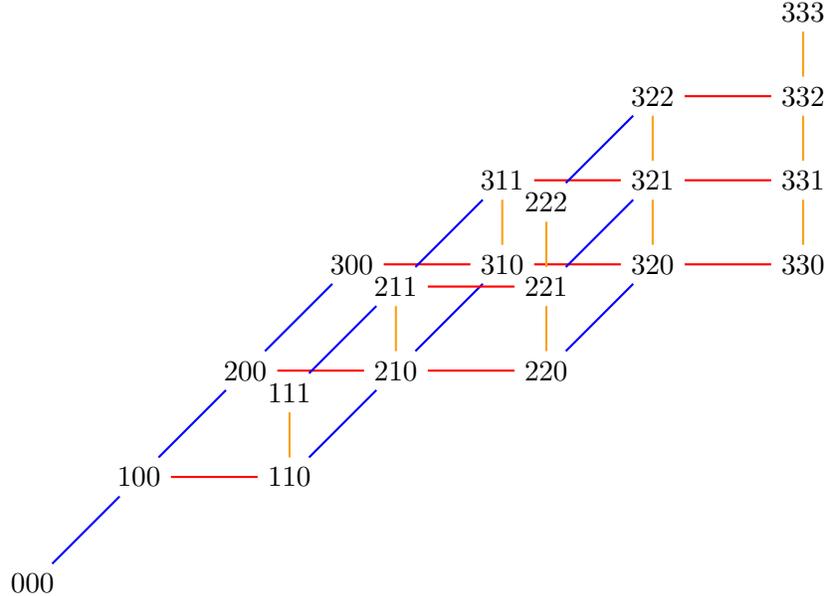
\begin{figure}[h]
\begin{center}
\begin{tikzpicture}[node distance = 20mm and 20mm,every edge quotes/.style = {auto, font=\footnotesize, sloped}]
    % Rank 0
    \node (000) at (0,0) {000};

    % Rank 1
    \node [above right of=000] (100) {100};
    \node [above right of=100] (200) {200};
    \node [above right of=200] (300) {300};
    \draw [blue, thick] (000) -- (100);
    \draw [blue, thick] (100) -- (200);
    \draw [blue, thick] (200) -- (300);

    \node [right of=100] (110) {110};
    \node [above right of=110] (210) {210};
    \node [above right of=210] (310) {310};
    \draw [blue, thick] (110) -- (210);
    \draw [blue, thick] (210) -- (310);

    \draw[red,thick] (100) -- (110);
    \draw[red,thick] (200) -- (210);
    \draw[red,thick] (300) -- (310);

    \node [right of=210] (220) {220};
    \node [above right of=220] (320) {320};
    \draw [blue, thick] (220) -- (320);

    \draw[red,thick] (210) -- (220);
    \draw[red,thick] (310) -- (320);

    \node [right of=320] (330) {330};
    
    \draw[red,thick] (320) -- (330);

    \node[above=6mm of 110] (111) {111};
    \node [above right of=111] (211) {211};
    \node [above right of=211] (311) {311};
    \draw [blue, thick] (111) -- (211);
    \draw [blue, thick] (211) -- (311);

    \draw[YellowOrange, thick] (110) -- (111);
    \draw[YellowOrange,thick] (210) -- (211);
    \draw[YellowOrange,thick] (310) -- (311);

    \node [right of=211] (221) {221};
    \node [above right of=221] (321) {321};
    \draw [blue, thick] (221) -- (321);

    \draw[red,thick] (211) -- (221);
    \draw[red,thick] (311) -- (321);

    \draw[YellowOrange,thick] (220) -- (221);
    \draw[YellowOrange,thick] (320) -- (321);

    \node [right of=321] (331) {331};
    
    \draw[red,thick] (321) -- (331);
    \draw[YellowOrange,thick] (330) -- (331);
    
    \node[above=6mm of 221] (222) {222};
    \draw[YellowOrange,thick] (221) -- (222);

    \node [above right of=222] (322) {322};
    \draw [blue, thick] (222) -- (322);

    \draw[YellowOrange,thick] (221) -- (222);
    \draw[YellowOrange,thick] (321) -- (322);

    \node [right of=322] (332) {332};
    
    \draw[red,thick] (322) -- (332);
    \draw[YellowOrange,thick] (331) -- (332);

    \node[above=6mm of 332] (333) {333};
    \draw[YellowOrange,thick] (332) -- (333);

\end{tikzpicture}
\caption{The Hasse diagram of $V(4,3)$.}
\label{V(4,3)}
\end{center}
\end{figure}

Now, we note that an element $s\in\mathbf{BL}(\mathbf{m})$ can be specified uniquely by listing the values of $\Phi_s(i_j)$ for each entry $i_j$ in $s$.
We call the vector $\langle\Phi_s(1_1),\dots,\Phi_s(n_{m_n}) \rangle$ the inversion vector of $s$.

\begin{example} Let us construct the barcode $s\in\mathbf{BL}(4,3,3,3)$ with inversion vector
\[\langle0,0,3,7;0,1,4;0,1,3;0,0,0\rangle.\]
We start by listing the subsequence consisting of all entries with index $1$, since these must appear in order:
\[1_1 2_1 3_1 4_1.\]
We then insert all of the entries that have the largest label, which in this case is $4$.
These must appear at the very end since the subsequence $4_14_24_3$ must appear in order:
\[1_1 2_1 3_1 4_1 4_2 4_3.\]
Next, we insert all of the entries with label $3$.
Observe that $3_2$ must appear after $4_1$ so that we have one entry (namely, $4_1$) between $3_1$ and $3_2$ whose label is greater than $3$.
Similarly, $3_3$ must appear after $4_1, 4_2$, and $4_3$:
\[1_1 2_1 3_1 4_1 3_2 4_2 4_3 3_3.\]
We insert $2_2$ after $3_1$, and we insert $2_3$ after $3_1,4_1,3_2$, and $4_2$:
\[1_1 2_1 3_1 2_2 4_1 3_2 4_2 2_3 4_3 3_3.\]
Finally, we insert the $1$s.
Observe that $1_2$ gets inserted directly after $1_1$, $1_3$ is inserted after $2_1,3_1$ and $2_2$, and $1_4$ is inserted after $2_1,3_1,2_2,4_1,3_2,4_2$, and $2_3$, giving us
\[s=1_1 1_2 2_1 3_1 2_2 1_3 4_1 3_2 4_2 2_3 1_4 4_3 3_3. \]
\hfill $\blacksquare$
\end{example}

Going from the permutation to inversion vector requires us to simply count, for each $i_j$, how many entries to the left of $i_j$ have labels greater than $i$.
This gives us a bijection $f$ between $\mathbf{BL}(\mathbf m)$ and the corresponding set of inversion vectors, which we claim is $V(\mathbf m)$.
To see why, notice that we get the bound
\[0=x_{i_1}\leq x_{i_j}\leq x_{i_{j+1}}\leq\sum_{k=i+1}^nm_k,\]
for all $i$ and $j$, from the fact that in an element $s\in\mathbf{BL}(\mathbf m)$, each entry with label $i$ other than $i_1$ can be placed to the right of at most $\sum_{k=i+1}^nm_k$ entries with label greater than $i$.

Further note that $f:\mathbf{BL}(\mathbf m)\rightarrow V(\mathbf m)$ is not an isomorphism, but does preserve the number elements of each rank.
This is because the rank of a barcode $s\in\mathbf{BL}(\mathbf m)$ is the number of inversions in $s$, i.e., 
\[\mathrm{rk}(s)=\sum_{i_j\in s}\Phi_s(i_j),\]
which is exactly the sum of the entries in $f(s)$, equivalently, the rank of $f(s)$ in $V(\mathbf m)$.
This tells us that $\mathbf{BL}(\mathbf m)$ and $V(\mathbf m)$ have the same rank-generating function.

\begin{theorem}\label{thm:rank_gen_func}
If $\mathbf m=(m_1,\dots,m_n)$, the rank-generating function of $V(\mathbf m)$, and thus of $\mathbf{BL}(\mathbf m)$, is
\[\prod_{i=1}^{n}\sqbinom{(\sum_{j=i}^nm_j)-1}{m_i-1}_q,\]
where $\sqbinom n k_q$ is the standard $q$-analog of the binomial coefficient:
\[\sqbinom nk_q= \frac {[n]_q!}{[n-k]_q![k]_q!}.\]
\end{theorem} 

\begin{proof}
For a vector $v\in V(\mathbf m)$ the rank of $v$ is $\sum_i v_i$.
We break this sum into blocks:
\[B_i=\sum_{j=\sum_{k=1}^{i-1}m_k}^{\sum_{k=1}^{i}m_k}v_j,\]
so that the rank of $v$ is $\sum_{i=1}^nB_i$.
Since $v$ must satisfy $0=v_{i_1}\leq v_{i_j}\leq v_{i_{j+1}}\leq\sum_{j=i+1}^nm_j$ for each $i,j$, the number of vectors $v$ with $\mathrm{rank}(v)=r$ is the number of ways to partition each block $B_i$ into at most $m_i-1$ parts each of size at most $\sum_{j=i+1}^nm_j$ (following the notation in \cite[Section 1.7]{EC1}):
\[\sum_{\sum_i B_i=r}p(B_i,\sum_{j=i+1}^nm_j,m_i-1).\]
Thus, the rank-generating function of $V(\mathbf m)$ is
\[\sum_k\left(\sum_{\sum B_i=k}p(B_i,\sum_{j=i+1}^nm_j,m_i-1)\right)q^k.\]
Since the coefficients are sums taken over compositions of the exponents, we can be decompose the summation above as a product of generating functions. 
Separating this generating function into a product gives us:
\[\prod_{i=1}^n \sum_k(p(k,\sum_{j=i+1}^nm_j,m_i-1))q^k.\]
Using \cite[Proposition 1.7.3]{EC1}, namely that 
\[\sqbinom {a+b}b_q=\sum_{\lambda\subseteq a\times b}q^{|\lambda|}=\sum_{k}p(k,a,b)q^k,\]
we can rewrite our product for the rank-generating function of $V(\mathbf m)$ as
\[\prod_{i=1}^n \sqbinom{m_i-1+\sum_{j=i+1}^nm_j}{m_i-1}_q,\]
as was to be proved.
\end{proof}

We now have a description of the rank-generating function of $\mathbf{BL}(\mathbf m)$ for any $m$.
We will now shift our attention to another combinatorial property of $\mathbf{BL}(\mathbf m)$.

%%%%%%%%%%%%%%%%%%%%%%%%%%%%%%%%%%%%%%%%%%%%%

\section{Maximal Chains}\label{sec:max_chains}

To find the number of maximal chains in the combinatorial barcode lattice, recall that $\mathbf{BL}(\mathbf m)$ is isomorphic to the principal order ideal generated by the permutation $\beta$ as established in Proposition \ref{principal ideal}.

For a permutation $\sigma=\sigma_1\cdots\sigma_M\in S_M$, define:
\begin{itemize}
    \item $r_i(\sigma)$ to be the number of entries $\sigma_j$ that appear before $\sigma_i$ such that $\sigma_j>\sigma_i$, and

    \item  $s_i(\sigma)$ to be the number of entries $\sigma_j$ that appear after $\sigma_i$ such that $\sigma_j<\sigma_i$.
\end{itemize} 
That is, we define
\begin{align*}
    r_i(\sigma)&=\#\{j\mid j<i,\sigma_j>\sigma_i\}\\
    s_i(\sigma)&=\#\{j\mid j>i,\sigma_j<\sigma_i\}.
\end{align*}

It is a well-known theorem of Stanley \cite[Corollary 4.2]{STANLEY1984359} that if the Ferrers diagram $\lambda(\sigma)$ obtained by reading the positive $r_i(\sigma)$'s for $i\in[M]$ in decreasing order is the same as the transpose of the Ferrers diagram $\mu(\sigma)$ obtained by reading the positive $s_i(\sigma)$'s for $i\in[M]$ in decreasing order, then the number of reduced decompositions of $\sigma$ is equal to the number of standard Young tableaux with shape $\lambda(\sigma)$.
By applying the hook length formula \cite[Corollary 7.21.6]{EC2} we obtain that the number of reduced decompositions of $\sigma$ is
\[\frac{|\lambda(\sigma)|!}{\prod_{(i,j)\in\lambda(\sigma)}h_{i,j}},\]
where $h_{i,j}$ is the number of cells $(k,l)\in\lambda(\sigma)$ with $k\geq i$ and $l\geq j$.

By \cite[Lemma 2.3]{LascouxSchutnzenberger} $\lambda(\sigma)=\mu'(\sigma)$ is equivalent to $\sigma$ being vexillary (i.e., $2143$-avoiding).
We now confirm that the permutation $\beta$ that generates the combinatorial barcode lattice is vexillary.
For the sake of contradiction, let $i,j,$ and $k$ be integers between $1$ and the length of $\beta$ such that $i<j<k$, while $\beta(j)<\beta(i)<\beta(k)$.
Keeping in mind the definition of $\beta$, suppose $\beta(i)$ belongs to block $B_t$.
Since the entries of each block appear in increasing order, $\beta(j)<\beta(i)$ and $i<j$ implies $\beta(j)$ must be from a block $B_s$ with $s<t$, and it cannot be the first element of the block.
Note, however, that this means every entry of $\beta$ after $\beta(j)$ must be lower than $\beta(i)$ since all of the entries of $\beta$ that appear at or after the second entry of the block $B_s$ must be from some block $B_r$ with $r\leq s$.
Thus, there cannot be a $k>j$ with $\beta(k)>\beta(i)$, meaning $\beta$ is $213$-avoiding. In particular, $\beta$ is $2143$-avoiding.
\\

\noindent The following theorem follows from the above exposition. 

\begin{theorem}\label{thm:max_chains}
If $\lambda$ is the Ferrers diagram obtained by ordering the positive entries of $\{\{r_i(\beta)\mid i\in[M]\}\}$ in decreasing order, where $\beta$ is the permutation generating $\mathbf{BL}(\mathbf m)$, then the number of maximal chains in $\mathbf{BL}(\mathbf m)$ is
    \[\frac{|\lambda|!}{\prod_{(i,j)\in\lambda}h_{i,j}},\]
    where $h_{i,j}$ is the number of cells $(k,l)\in\lambda(\sigma)$ with $k\geq i$ and $l\geq j$.
\end{theorem}

When we set $\mathbf{m}=(k^n)$, we obtain the following corollary.

\begin{corollary}
The number of maximal chains in $\mathbf{BL}(k^n)$ is
\[n!\left(\prod_{i=1}^{n-1}\left(\prod_{j=1}^{2k-2}(((2k-1)(n-i))+j)^{\min\{j,2k-1-j\}}\right)^i\right)^{-1}\]
and in particular the number of maximal chains in $\mathbf{BL}(2^n)$ is
\[(n(n-1))!\prod_{i=1}^{n-1}\frac{3^i(i!)}{(3i)!}.\]
\end{corollary}

\begin{proof} 
We can check that the permutation $\beta$ generating $\mathbf{BL}(k^n)$ has $r_i(\beta)=0$ for $i\leq n$ and $r_i(\beta)=k\lfloor\frac{i-n-1}{k-1}\rfloor$ for $n+1\leq i\leq kn$.
It follows that the resulting diagram $\lambda$ is \[\lambda=((k(n-1))^{k-1},(k(n-2))^{k-1},\dots,k^{k-1}).\]
Let $H=(H_{ij})$ be the filling of the cells of $\lambda$ with their hook length (i.e., $H_{ij}=h_{ij}$).
Then the number of maximal chains in $\mathbf{BL}(\mathbf m)$ is $n!$ divided by the product of the cells in the filling $H$.
To simplify calculations, note that the diagram can be partitioned into $\binom n2$ rectangles of size $(k-1)\times k$ by placing the top left cell of a $(k-1)\times k$ rectangle at cell $(i(k-1)+1,jk+1)$ of $\lambda$ for each $i$ and $j$ for which the cell is contained in $\lambda$.
(See Figure \ref{ferrersDiag} for an example.)

\begin{figure}
    \centering
    \ytableausetup{nosmalltableaux}
    \begin{ytableau}
    *(yellow)14&*(yellow)13&*(yellow)12&*(magenta) 9&*(magenta) 8&*(magenta)7 &*(cyan)4 &*(cyan)3 &*(cyan) 2\\
    *(yellow)13&*(yellow)12&*(yellow)11&*(magenta) 8&*(magenta) 7&*(magenta)6 &*(cyan)3 &*(cyan)2 &*(cyan) 1\\
    *(magenta) 9&*(magenta) 8&*(magenta)7 &*(cyan)4 &*(cyan)3 &*(cyan) 2\\
    *(magenta) 8&*(magenta) 7&*(magenta)6 &*(cyan)3 &*(cyan)2 &*(cyan) 1\\
    *(cyan) 4& *(cyan) 3& *(cyan) 2\\
    *(cyan) 3&*(cyan) 2&*(cyan) 1
\end{ytableau}
    \caption{The filling $H$ of the Ferrers diagram $\lambda(\beta)$ for $\beta$ generating $\mathbf{BL}(3^4)$, partitioned into rectangles of size $2\times3$.
    \label{ferrersDiag}}
\end{figure}

With this in mind, we define $R_{ij}$ for $2\leq i+j\leq n$ to be the $(k-1)\times k$ array of numbers where the element of $R_{ij}$ in row $a$ and column $b$ is the hook length (in $\lambda$) of the cell in the $a^{th}$ row and $b^{th}$ column of the rectangle starting at cell $((i-1)(k-1)+1,(j-1)k+1)$ of $\lambda$.
More succinctly, we let $R_{ij}:[k-1]\times[k]\rightarrow \mathbb N$ be given by
\[R_{ij}(a,b)=H_{((i-1)(k-1)+a,(j-1)k+b)}.\]
Notice that for all $(i,j)$ and $(i',j')$ with $i+j=i'+j'\leq n$ we have $R_{ij}=R_{i'j'}$, and because of this we need only keep track of rectangles in the top row.
Letting $R_i=R_{i1}$, 
since the rectangle corresponding to $R_i$ occurs $i$ times in $\lambda$, we have that the product of the hook lengths is
\[\prod_{i=1}^{n-1}\left(\prod_{(a,b)\in[k-1]\times[k]}R_i(a,b)\right)^i.\]

Next, we calculate the $R_i(a,b)$.
Define $\mathrm{Hook}(a,b)$ to be the hook starting at $(a,b)$ in $\lambda$, then observe that $\#(\mathrm{Hook}(a,b)\setminus R)$ is the same for any $(a,b)\in R$.
Thus, we have \[R_i(a,b)=1+R_i(a+1,b)=1+R_i(a,b+1),\] which also gives us \[R_i(a,b)=R_i(a-1,b+1)=R_i(a+1,b-1).\]
From these two equalities, we can conclude that there are $(k-1)+k-1=2k-2$ distinct values in the cells of $R_i$ since we have
\begin{align*}
R_i(1,1)&=1+R_i(2,1)\\
&\vdots\\
&=(k-3)+R_i(k-2,1)\\
&=(k-2)+R_i(k-1,1)\\
&=(k-1)+R_i(k-1,2)\\
&\vdots\\
&=(2k-3)+R_i(k-1,k),
\end{align*}
and every other cell in the rectangle is determined by these $2k-2$ values.
Let us relabel once again and take $R_i(j)$ for $1\leq j\leq2k-2$ to be
\[R_i(j)=\begin{cases}
    R_i(j,1)& \text{ if } j\leq k-1,\\
    R_i(k-1,j-k+2)& \text{ if } k\leq j.
\end{cases}\]
The cells $(a,b)\in [k-1]\times[k]$ such that $R_i(a,b)=R_i(j)$ appear in a diagonal of length $\min\{j,2k-1-j\}$.
It follows that the product of the hook lengths is
\[\prod_{i=1}^{n-1}\left(\prod_{j=1}^{2k-2}(R_i(k-1,k)+j)^{\min\{j,2k-1-j\}}\right)^i.\]
Now, note that \[R_{n-1}(k-1,k)=1 \text{ and } R_{i}(k-1,k)=R_{i+1}(k-1,k)+2k-1,\] since there are $k$ more cells to the right of $R_i(j)$ than there are to the right of $R_{i+1}(j)$ and there are $k-1$ more cells below $R_i(j)$ than there are below $R_{i+1}(j)$ for all $j$.
Thus, we can finally conclude that the product of the hook lengths of cells in $\lambda$ is
\[\prod_{i=1}^{n-1}\left(\prod_{j=1}^{2k-2}(((2k-1)(n-i))+j)^{\min\{j,2k-1-j\}}\right)^i,\]
and the desired result follows from the hook length formula.
\end{proof}
%%%%%%%%%%%%%%%%%%%%%%%%%%%%%%%%%%%%%%%

\section{Connection to matchings}\label{sec:matchings}

Jaramillo-Rodriguez in \cite{jaramillorodriguez2022barcode} was mainly interested in barcodes as combinatorial invariants on the space of barcodes, but it is interesting to note that the underlying set of $\mathbf{BL}(2^n)$ is the set of perfect matchings on $[2n]$.
This introduces a new order on the set of perfect matchings, and the following two results suggest that this may be a natural way to order the matchings.

\begin{proposition}
The permutational matchings on $[2n]$, i.e., the matchings that avoid the pattern $1122$, consist of exactly the matchings in the interval \[[12\cdots n12\cdots n,12\cdots nn\cdots 21]\subseteq \mathbf{BL}(2^n).\]
\end{proposition}

\begin{proof}
Recall that the first appearances of each label any matching $s\in\mathbf{BL}(\mathbf 2^n)$ must occur in increasing order.
(That is, $s$ must contain the subsequence $1_12_1\cdots n_1$ in that order).
Now note that an occurrence of the pattern $1122$ in $s$ corresponds to the presence of the subsequence $a_1a_2b_1b_2$ in $s$ for some $a,b\in[n]$.
Such a subsequence occurs exactly when there is an occurrence of some entry $a_2$ before an occurrence of some entry $b_1$, for $a$ and $b$ in $[n]$, so we observe that a matching $s$ is permutational if and only if it is of the form $\tau=1_12_1\cdots n_1\sigma$, where $\sigma$ is a permutation of $\{1_2,\dots,n_2\}$.

It follows that a matching $s$ is permutational if and only if its inversion set contains $\{(i_1,j_2)\mid j<i\in[n]\}$.
Since $12\cdots n12\cdots n$ is the matching having exactly this inversion set, and since the weak Bruhat order on the symmetric group is the same an inclusion of inversion sets, we have that $s\geq12\cdots n12\cdots n$ if and only if $s$ is permutational.
\end{proof} 

\begin{corollary}\label{permutational cor}
There are $n!$ elements in the interval $[12\cdots n12\cdots n,12\cdots nn\cdots 21]$. 
\end{corollary}

\begin{proof}
This follows directly from the previous observation and the fact that there are $n!$ permutations of $[n]$ and thus $n!$ permutational matchings.
\end{proof}

\begin{proposition}
The non-nesting matchings on $[2n]$, i.e., the matchings that avoid the pattern $1221$, consist of exactly the matchings in the interval \[[1122\cdots nn,12\cdots n12\cdots n]\subseteq \mathbf{BL}(\mathbf m).\]
\end{proposition} 

\begin{proof}
A matching $s$ is non-nesting if and only if the subsequence of $s$ consisting of only the entries $1_1,2_1,\dots,n_1$ has the same relative order as the subsequence of $s$ consisting of only the entries $1_2,2_2,\dots,n_2$.
Since matchings are combinatorial barcodes, we have that these subsequences must be appear in increasing order, so $s$ is a non-nesting matching if and only if the only inversions of $s$ are of the form \[(i_1,j_2) 
\text{ for } j<i.\]
This is equivalent to $s$ being less than or equal to $12\cdots n12\cdots n$ in the weak Bruhat order, since the inversion set of $12\cdots n12\cdots n$ is $\{(j_1,i_2)\mid j<i\}$.
\end{proof} 

\begin{corollary}\label{nonnesting cor} The size of the interval $[1122\cdots nn,12\cdots n12\cdots n]$ is the $n^{th}$ Catalan number, 
\[C_n=\frac{1}{n+1}\binom{2n}{n}.\]
\end{corollary}

\begin{proof}
It is well-known that the non-nesting matchings on $[2n]$ are counted by the Catalan numbers, so this follows from the previous observation.      
\end{proof}

%%%%%%%%%%%%%%%%%%%%%%%%%%%%%%%%%%%%%%%

\section{Further directions}\label{sec:further_directions}

Section \ref{sec:matchings} suggests that it may be worth looking into intervals in the combinatorial barcode lattice from the perspective of matchings.
The order dimension of the combinatorial barcode lattice is currently unknown.
Flath characterized the order dimension of the multinomial Newman lattice in \cite{OrderDimOfLm} using techniques from formal context analysis.
Flath's proof does not immediately apply to the combinatorial barcode lattice, in part because the meet-irreducible elements of the combinatorial barcode lattice do not have as simple a characterization as the meet-irreducible elements of the multinomial Newman lattice, but we are hopeful that Flath's proof could be modified to give the order dimension of the barcode lattice.

Also, recall that we saw that in general
\[\sum_{\sigma\in\mathbf{BL}(\mathbf m)}q^{\mathrm{inv}(\sigma)}\neq\sum_{\sigma\in\mathbf{BL}(\mathbf{m})}q^{\mathrm{maj}(\sigma)}.\]
The left side is the rank-generating function, which we proved is
$\prod_{i=1}^{n}\sqbinom{(\sum_{j=i}^nm_j)-1}{m_i-1}_q$,
but we do not yet have a description of the right hand side.
Combinatorial interpretations of the inversions and major index statistics in terms of the properties of the barcodes would also be of interest.

Parts of the work in this paper can be generalized to other classes of posets, particularly those of the form $\mathcal L(P,\omega)$ mentioned in section \ref{sec:initial_results}.
A first step may be to determine for which $(P,\omega)$ can we use similar proofs as to those in Sections \ref{sec:rank_gen_func} and \ref{sec:max_chains} to find the rank-generating function and number of maximal chains in $\mathcal L(P,\omega)$.

%%%%%%%%%%%%%%%%%%%%%%%%%%%%%%%%%%%%%%%

\section*{Acknowledgments}
The authors thank Elena Jaramillo-Rodriguez for her fruitful correspondence. 
ARVM is partially supported by the NSF under Award DMS-2102921. 
This material is also based in part upon work supported by the NSF Grant DMS-1928930 and the Alfred P. Sloan Foundation under grant G-2021-16778, while ARVM was in residence at the Simons Laufer Mathematical Sciences Institute in Berkeley, California, during the Fall 2023 semester.

%%%%%%%%%%%%%%%%%%%%%%%%%%%%%%%%%%%%%%%
\bibliographystyle{amsplain}
\bibliography{b}

%%%%%%%%%%%%%%%%%%%%%%%%%%%%%%%%%%%%%%%

\end{document}